\documentclass[11pt]{article}

\usepackage[noquiver]{kyrem1sty}

\usepackage[backend=biber,style=alphabetic,doi=false,isbn=false,url=false,eprint=false]{biblatex}
\hypersetup{%
  pdftitle={The Area Asymptotics of (sn,n)-Dyck Paths},%
  pdfauthor={Evan Conway and James Harbour},%
  colorlinks=true,%
  linkcolor=blue,%
  citecolor=blue,%
  filecolor=blue,%
  menucolor=blue,%
  urlcolor=blue,%
  pdfnewwindow=true,%
  pdfstartview=FitBH
}   

\usepackage{tikz}
\usetikzlibrary{snakes, arrows.meta,shapes,patterns,calc,decorations.pathreplacing,positioning,decorations.markings}

\usepackage{ytableau}
\usepackage{caption}

\newcommand{\DrawDyckPath}[3]{%
    \begin{tikzpicture}
        \draw[help lines, color=gray!30] (0,0) grid (#1*#2,#2);
        
        \draw[very thick, color=black] #3;
        \draw[dashed, very thick, color=black] (0,0) -- (#1*#2,#2);
    \end{tikzpicture}
}

\newcommand{\DrawShadedDyckPath}[3]{%
    \begin{tikzpicture}
        \draw[help lines, color=gray!30] (0,0) grid (#1*#2,#2);
        
        \draw[very thick, color=black] #3;
        \draw[dashed, very thick, color=black] (0,0) -- (#1*#2,#2);
        \fill[pattern=north west lines, pattern color=gray] #3 -- cycle;
    \end{tikzpicture}
}

\usepackage{subfiles}

\usepackage{geometry}

\title{The Area Asymptotics of $(sn,n)$-Dyck Paths}
\author{Evan Conway \and James Harbour}
\date{}

\begin{document}
\maketitle

\begin{abstract}
We study the total and average area of $(sn,n)$-Dyck paths: lattice paths from $(0,0)$ to $(sn,n)$ that stay weakly below the line $x=sy$, counted by the Fuss-Catalan numbers. Generalizing a result of Merlini, Sprugnoli, and Verri for the case $s=1$, we derive an exact formula for the total area over all such paths. From this we obtain explicit upper and lower bounds for both the total and the average area, together with the corresponding asymptotics: for fixed $s$, the average area is asymptotic to $\sqrt{\pi s(s+1)/8} \cdot n^{3/2}$, while for fixed $n$ it is asymptotic to $sn \cdot Q(n)/2$ as $s$ grows large, where $Q(n)$ denotes Ramanujan's $Q$-function. Along the way, we confirm a conjecture of Kotesovec on the asymptotics of a binomial sum that also arises in several other enumeration problems.
\end{abstract}



\section{Introduction}

\subsection{Motivation}

Dyck paths---lattice paths that go from $(0,0)$ to $(n,n)$ and lie weakly below the line $x=y$---are a fundamental class of lattice paths counted by the Catalan numbers, one of the most ubiquitous sequences in combinatorics. The Catalan numbers enumerate numerous combinatorial structures, including binary trees, parenthesizations, and noncrossing partitions, with Dyck paths providing a natural geometric interpretation. This deep connection has made Dyck paths a central object of study across various mathematical disciplines.

One area of intense interest has been the statistics of Dyck paths: numerical properties, such as the area, that can be derived from these paths. The study of these statistics has been prominent in a variety of fields, such as algebraic combinatorics (\cite{haglund2008q}), probability (\cite{bernoulli_excursion}), and computer science (\cite{banderier2006area}).

Another common direction of research has been various variants of Dyck paths. In particular, one of the most studied variants has been $(sn,n)$-Dyck paths, paths that go from $(0,0)$ to $(sn,n)$ and lie weakly beneath the line $x=sy$. One major advantage of this generalization is that $(sn,n)$-Dyck paths are enumerated by the Fuss-Catalan numbers, which are perhaps the most natural generalization of the Catalan numbers (\cite{armstrong2007generalizednoncrossingpartitionscombinatorics}). These numbers have been studied for a surprisingly long time, with the first paper on the topic being published in 1791 (\cite{przytycki1998polygon}). This connection makes $(sn,n)$-Dyck paths a valuable tool for studying properties of objects enumerated by the Fuss-Catalan numbers, such as $k$-ary trees or polygon decompositions.

In this paper, we combine elements from these two areas of research by investigating the area of $(sn,n)$-Dyck paths. As part of this, we seek to generalize prior area results for Dyck paths from Merlini et al. (\cite{merlini}), which established an exact formula and asymptotics for the total area over all Dyck paths with fixed $n$.

The area below lattice paths has also been studied using analytic combinatorics. Most relevantly, Banderier and Gittenberger \cite{banderier2006area} use the kernel method to derive generating functions and asymptotics for the area below directed lattice paths, a class that includes $(sn,n)$-Dyck paths. In particular, their results provide an alternative route to the area generating function that we derive in Section \ref{sec:total_area_formula}.

Our research also draws connections to previous investigations into properties of polyominoes (\cite{mansour2021enumerations}) and the tennis ball problem (\cite{MERLINI2002307}); in several cases, we find ways to simplify and extend these related results. In addition, our asymptotics for large $s$ reveal a connection to Ramanujan's $Q$-function and the classical birthday problem (\cite{flajolet1995ramanujan}).

\subsection{Main Results}

Our first main result is an exact formula for the total area $A^s_n$ of $(sn,n)$-Dyck paths, mirroring the known formula for the $s = 1$ case. We prove it in Theorem \ref{thm:exact_area} by combining two existing touchpoint decompositions: the one used by Merlini et al. to find the total area of Dyck paths (\cite{merlini}, Theorem 2.1), and the one used by Krattenthaler to count $(sn,n)$-Dyck paths (\cite{krattenthaler2015lattice}, second proof of Theorem 10.4.5).
\begin{equation}\begin{split}
    A^s_n = \frac{1}{2} \lr{\sum_{k=0}^{n} {(s+1)k \choose k} {(s+1)(n-k) \choose n-k} - {(s+1)n+1 \choose n}}
\end{split}\end{equation}

Our second main result consists of explicit upper and lower bounds for both the total area (Theorem \ref{thm:total_area_bounds}) and the average area (Theorem \ref{thm:average_area_bounds}), valid for all $n \geq 2$. These bounds immediately yield the asymptotics in $n$ of both quantities (Corollaries \ref{cor:total_n_asymp} and \ref{cor:average_n_asymp}); in particular, writing $F^s_n$ for the $n$-th Fuss-Catalan number, the average area satisfies
\begin{equation}\begin{split}
    \frac{A^s_n}{F^s_n} \sim \frac{1}{2} \sqrt{\frac{\pi}{2}} \sqrt{s(s+1)} \cdot n^{\frac{3}{2}}
\end{split}\end{equation}
As part of the derivation, we prove bounds and asymptotics for the sum $\sum_{k=0}^{n} {(s+1)k \choose k} {(s+1)(n-k) \choose n-k}$, an expression that appears surprisingly often in the literature (see \cite{MERLINI2002307} for one interesting case, and \cite{oeisA006256} for other examples); this resolves a conjecture of Vaclav Kotesovec.

Our final main result concerns the opposite regime, where $n$ is fixed and $s$ grows large. We find bounds for the total and average area suited to this regime (Lemmas \ref{lem:total_area_bounds_s} and \ref{lem:average_area_bounds_s}), leading to the asymptotics in $s$ of the average area (Corollary \ref{cor:average_s_asymp}):
\begin{equation}\begin{split}
    \frac{A^s_n}{F^s_n} \sim \frac{1}{2} sn \cdot Q(n)
\end{split}\end{equation}
where $Q(n)$ is Ramanujan's $Q$-function. This connects the average area to the classical birthday problem, as discussed in Section \ref{sec:s_average_area}.

\section{Background}

\subsection{$(n,n)$-Dyck Paths}

\begin{definition}[Dyck Path]
A Dyck path is a walk in the $(x,y)$ plane from $(0,0)$ to $(n,n)$ with steps $(1,0)$ or $(0,1)$ that lies weakly beneath the line $x=y$. We use $D_n$ to denote the set of Dyck paths.
\end{definition}

\begin{remark}
We will also call these $(n,n)$-Dyck paths to distinguish them from the $(sn,n)$-Dyck paths that we will introduce later.
\end{remark}

\begin{example} For $n = 3$, we have the following five Dyck paths

\begin{minipage}{0.19\textwidth}
\DrawDyckPath{1}{3}{(0,0) -- (1,0) -- (1,1) -- (2,1) -- (2,2) -- (3,2) -- (3,3)}
\end{minipage}%
\begin{minipage}{0.19\textwidth}
\DrawDyckPath{1}{3}{(0,0) -- (1,0) -- (1,1) -- (2,1) -- (3,1) -- (3,2) -- (3,3)}
\end{minipage}%
\begin{minipage}{0.19\textwidth}
\DrawDyckPath{1}{3}{(0,0) -- (1,0) -- (2,0) -- (2,1) -- (2,2) -- 
(3,2) -- (3,3)}
\end{minipage}%
\begin{minipage}{0.19\textwidth}
\DrawDyckPath{1}{3}{(0,0) -- (1,0) -- (2,0) -- (2,1) -- (3,1) -- 
(3,2) -- (3,3)}
\end{minipage}%
\begin{minipage}{0.19\textwidth}
\DrawDyckPath{1}{3}{(0,0) -- (1,0) -- (2,0) -- (3,0) -- (3,1) -- 
(3,2) -- (3,3)}
\end{minipage}
\end{example}

\begin{remark}
The number of Dyck paths for a particular $n$ is given by the Catalan numbers $C_n$, which are defined as
\begin{equation}\begin{split}
    C_n := \frac{1}{n+1} {2n \choose n}
\end{split}\end{equation}
\end{remark}

In addition to enumerating Dyck paths, the Catalan numbers count a wide variety of mathematical objects. Stanley's \textit{Enumerative Combinatorics: Volume 2} (\cite{stanley1999enumerative}) lists 66 such interpretations, and many other interpretations can be found in the literature.

\begin{definition}[Dyck Path Area]
The area of a Dyck path $P\in D_n$, denoted $\area(P)$, is the total area between the path and the line $x=y$.
\end{definition}

\begin{example}
For $n=3$, we have the following five Dyck path areas

\begin{minipage}{0.19\textwidth}
\centering
\DrawShadedDyckPath{1}{3}{(0,0) -- (1,0) -- (1,1) -- (2,1) -- (2,2) -- (3,2) -- (3,3)}
Area = 1.5
\end{minipage}%
\begin{minipage}{0.19\textwidth}
\centering
\DrawShadedDyckPath{1}{3}{(0,0) -- (1,0) -- (1,1) -- (2,1) -- (3,1) -- (3,2) -- (3,3)}
Area = 2.5
\end{minipage}%
\begin{minipage}{0.19\textwidth}
\centering
\DrawShadedDyckPath{1}{3}{(0,0) -- (1,0) -- (2,0) -- (2,1) -- (2,2) -- (3,2) -- (3,3)}
Area = 2.5
\end{minipage}%
\begin{minipage}{0.19\textwidth}
\centering
\DrawShadedDyckPath{1}{3}{(0,0) -- (1,0) -- (2,0) -- (2,1) -- (3,1) -- (3,2) -- (3,3)}
Area = 3.5
\end{minipage}%
\begin{minipage}{0.19\textwidth}
\centering
\DrawShadedDyckPath{1}{3}{(0,0) -- (1,0) -- (2,0) -- (3,0) -- (3,1) -- (3,2) -- (3,3)}
Area = 4.5
\end{minipage}
\end{example}

\begin{remark}
In addition to our definition of the area, there is another popular definition (commonly used in algebraic combinatorics) that counts only the number of full squares between the Dyck path and the center line $x=y$. We can easily convert between these two definitions by noting that the only non-square areas are the triangles next to the center line, which have total area $\frac{n}{2}$. Subtracting this amount from the area given by our definition yields the area given by the other definition.
\end{remark}

\subsection{$(sn,n)$-Dyck Paths}

There are several ways to extend Dyck paths, but in this paper we will specifically focus on what we call $(sn,n)$-Dyck paths.

\begin{definition}[$(sn,n)$-Dyck Path]
An $(sn,n)$-Dyck path is a walk in the $(x,y)$ plane from $(0,0)$ to $(sn,n)$ with steps $(1,0)$ or $(0,1)$ that lies weakly beneath the line $x=sy$. We use $D^s_n$ to denote the set of $(sn,n)$-Dyck paths.
\end{definition}

\begin{example} For $s = 2, n = 2$, we have the following three $(sn,n)$-Dyck paths

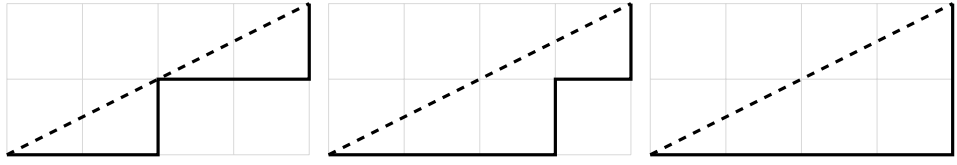
\begin{center}
\begin{minipage}{0.25\textwidth}
\DrawDyckPath{2}{2}{(0,0) -- (1,0) -- (2,0) -- (2,1) -- (3,1) -- (4,1) -- (4,2)}
\end{minipage}
\begin{minipage}{0.25\textwidth}
\DrawDyckPath{2}{2}{(0,0) -- (1,0) -- (2,0) -- (3,0) -- (3,1) -- (4,1) -- (4,2)}
\end{minipage}
\begin{minipage}{0.25\textwidth}
\DrawDyckPath{2}{2}{(0,0) -- (1,0) -- (2,0) -- (3,0) -- (4,0) -- (4,1) -- (4,2)}
\end{minipage}
\end{center}
\end{example}

The number of $(sn,n)$-Dyck paths for a particular $n$ is given by the Fuss-Catalan numbers $F^s_n$, which are defined as
\begin{equation}\begin{split}
    F^s_n := \frac{1}{sn+1} {(s+1)n \choose n}
\end{split}\end{equation}

The definition of the area proceeds similarly to that for $(n,n)$-Dyck paths
\begin{definition}[$(sn,n)$-Dyck Path Area]
The area of an $(sn,n)$-Dyck path is the total area between the path and the line $x=sy$.
\end{definition}

\begin{example} For $s = 2, n = 2$, we have the following three $(sn,n)$-Dyck path areas

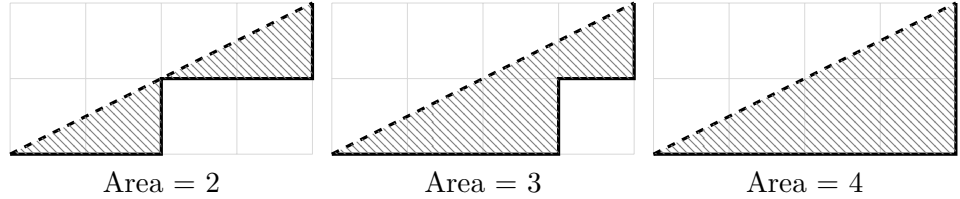
\begin{center}
\begin{minipage}{0.25\textwidth}
\centering
\DrawShadedDyckPath{2}{2}{(0,0) -- (1,0) -- (2,0) -- (2,1) -- (3,1) -- (4,1) -- (4,2)}
Area = 2
\end{minipage}
\begin{minipage}{0.25\textwidth}
\centering
\DrawShadedDyckPath{2}{2}{(0,0) -- (1,0) -- (2,0) -- (3,0) -- (3,1) -- (4,1) -- (4,2)}
Area = 3
\end{minipage}
\begin{minipage}{0.25\textwidth}
\centering
\DrawShadedDyckPath{2}{2}{(0,0) -- (1,0) -- (2,0) -- (3,0) -- (4,0) -- (4,1) -- (4,2)}
Area = 4
\end{minipage}
\end{center}
\end{example}

\subsection{Touchpoint Decomposition: $(n,n)$ Case}

To find a formula for the area of $(sn,n)$-Dyck paths, we will use the technique of first touchpoint decomposition, which has previously been used to find the area of $(n,n)$-Dyck paths \cite{merlini}. We will begin by exploring this technique's applications to the $(n,n)$ case.

The concept of a touchpoint decomposition relies on the observation that any Dyck path must eventually touch the diagonal at some point after the origin. We call these points touchpoints.

\begin{definition}[Touchpoints]
    A touchpoint of a Dyck path $P \in D_n$ is any point $Z \in P$, other than the origin, that lies on the center diagonal $x=y$.
\end{definition}

When looking at our possible touchpoints, we can observe that one of these touchpoints must be the first touchpoint (earliest in the path). In the case where there is only a touchpoint at the end, that is the first touchpoint. We can use this to split each Dyck path into two segments: the segment before the first touchpoint, and the segment after the first touchpoint.

\begin{example} For $n = 3$, we have the following five first touchpoint decompositions

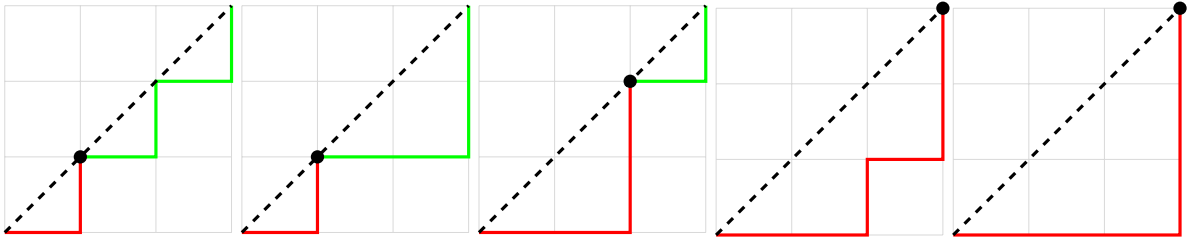
\begin{minipage}{0.19\textwidth}
\begin{tikzpicture}
    \draw[help lines, color=gray!30] (0,0) grid (3,3);
    
    \draw[very thick, color=red] (0,0) -- (1,0) -- (1,1);
    \draw[very thick, color=green] (1,1) -- (2,1) -- (2,2) -- (3,2) -- (3,3);
    \draw[dashed, very thick, color=black] (0,0) -- (3,3);

    \fill[black] (1,1) circle (2.5pt);
\end{tikzpicture}
\end{minipage}%
\begin{minipage}{0.19\textwidth}
\begin{tikzpicture}
    \draw[help lines, color=gray!30] (0,0) grid (3,3);
    
    \draw[very thick, color=red] (0,0) -- (1,0) -- (1,1);
    \draw[very thick, color=green] (1,1) -- (3,1) -- (3,3);
    \draw[dashed, very thick, color=black] (0,0) -- (3,3);

    \fill[black] (1,1) circle (2.5pt);
\end{tikzpicture}
\end{minipage}%
\begin{minipage}{0.19\textwidth}
\begin{tikzpicture}
    \draw[help lines, color=gray!30] (0,0) grid (3,3);
    
    \draw[very thick, color=red] (0,0) -- (2,0) -- (2,2);
    \draw[very thick, color=green] (2,2) -- (3,2) -- (3,3);
    \draw[dashed, very thick, color=black] (0,0) -- (3,3);

    \fill[black] (2,2) circle (2.5pt);
\end{tikzpicture}
\end{minipage}%
\begin{minipage}{0.19\textwidth}
\begin{tikzpicture}
    \draw[help lines, color=gray!30] (0,0) grid (3,3);
    
    \draw[very thick, color=red] (0,0) -- (2,0) -- (2,1) -- (3,1) -- (3,3);
    \draw[dashed, very thick, color=black] (0,0) -- (3,3);

    \fill[black] (3,3) circle (2.5pt);
\end{tikzpicture}
\end{minipage}%
\begin{minipage}{0.19\textwidth}
\begin{tikzpicture}
    \draw[help lines, color=gray!30] (0,0) grid (3,3);
    
    \draw[very thick, color=red] (0,0) -- (3,0) -- (3,3);
    \draw[dashed, very thick, color=black] (0,0) -- (3,3);

    \fill[black] (3,3) circle (2.5pt);
\end{tikzpicture}
\end{minipage}
\end{example}

We will now apply this technique in Figure \ref{fig:first_touchpoint_decomposition}, which displays a first touchpoint decomposition for some arbitrary Dyck path $P \in D_{n+1}$.

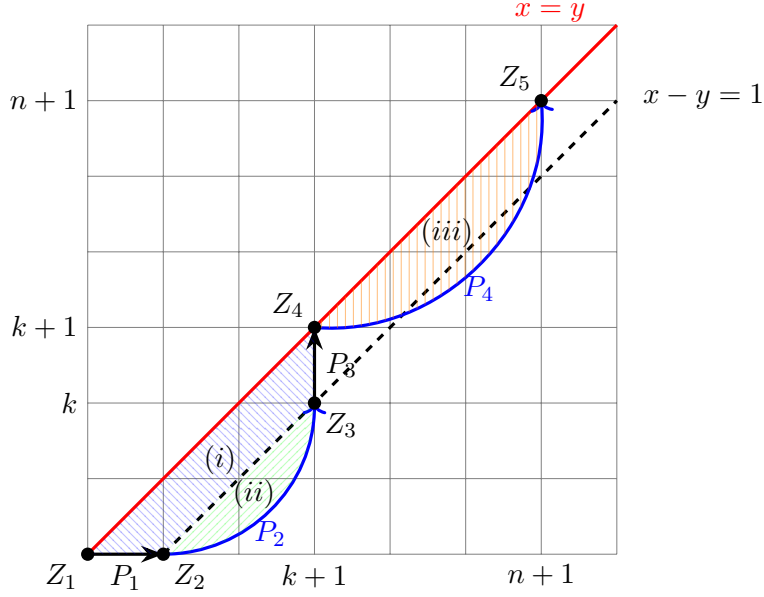
\begin{figure}[H]
    \centering
    \begin{tikzpicture}[scale=1, every node/.style={scale=1}]
        \def\n{5}
        \def\k{2}
        
        \draw[gray, very thin] (0,0) grid (\n+2,\n+2);
        
        \draw[red, very thick] (0,0) -- (\n+2,\n+2) node[anchor=south west] [xshift=-1.5cm,yshift=-0.1cm] {$x = y$};
        
        \coordinate (O) at (0,0);
        \coordinate (B) at (1,0);
        \coordinate (Q) at (\k+1,\k);
        \coordinate (ExtendedQ) at (\n+2,\n+1); 
        \coordinate (P) at (\k+1,\k+1);
        \coordinate (A) at (\n+1,\n+1);
        \coordinate (ExtendedA) at (\n+2,\n+2); 
        
        \draw[very thick, -Stealth] (O) -- (B) node[midway,below] {$P_1$};
        \draw[very thick, dashed] (B) -- (ExtendedQ) node[anchor=south west] [xshift=+0.20cm,yshift=-0.25cm] {$x-y=1$};
        \draw[very thick, -Stealth] (Q) -- (P) node[midway,right] {$P_3$};
        
        \draw[very thick, blue, ->] (B) to[bend right=45] node[midway,below] {$P_2$} (Q);
        \draw[very thick, blue, ->] (P) to[bend right=50] node[midway,below] {$P_4$} (A);
        
        \fill[pattern=north west lines, pattern color=blue, opacity=0.5] (O) -- (B) -- (Q) -- (P) -- cycle;
        \fill[pattern=north east lines, pattern color=green, opacity=0.5] (B) to[bend right=45] (Q) -- (Q) -- (B);
        \fill[pattern=vertical lines, pattern color=orange, opacity=0.5] (P) to[bend right=50] (A) -- (A) -- (P);
        
        \foreach \point in {O, B, Q, P, A}
            \fill [black] (\point) circle (2.5pt);
        
        \node[below left] at (O) {$Z_1$};
        \node[below right] at (B) {$Z_2$};
        \node[below right] at (Q) {$Z_3$};
        \node[above left] at (P) {$Z_4$};
        \node[above left] at (A) {$Z_5$};
        
        \coordinate (L) at  (2.6,0.4);
        \coordinate (W) at (5,4);
        \node at (barycentric cs:O=1,B=1,Q=1,P=1) {$(i)$};
        \node at (barycentric cs:B=1,Q=1,L=1) {$(ii)$};
        \node at (barycentric cs:P=1,A=1,W=2) {$(iii)$};
        
        \node[below] at (\k+1,0) {$k+1$};
        \node[below] at (\n+1,0) {$n+1$};
        \node[left] at (0, \k+1) {$k+1$};
        \node[left] at (0, \k) {$k$};
        \node[left] at (0, \n+1) {$n+1$};
        \end{tikzpicture}
    \caption{A first touchpoint decomposition}
    \label{fig:first_touchpoint_decomposition}
\end{figure}

In Figure \ref{fig:first_touchpoint_decomposition}, the first touchpoint is $Z_4 = (k+1, k+1)$. We can then analyze the three areas labeled in the figure, and examine their contribution to $\area(P)$:
\begin{enumerate}[(i)]
    \item We know that no part of $P$ will cross into this region, since otherwise $Z_4$ would not be the first touchpoint. This means that we can explicitly calculate this region's contribution to $\area(P)$, which is $k + \frac{1}{2}$.
    \item Note that the segment $P_2$ goes from $Z_2 = (1,0)$ to $Z_3 = (k+1, k)$ and will not cross the line $x = y+1$. This means that $P_2$ is a $(k,k)$-Dyck path. More precisely, we can write its area as $\area(P_2)$, with $P_2 \in D_k$.
    \item Similarly to (ii), we can note that $P_4$ is an $(n-k,n-k)$-Dyck path, so we can write its area as $\area(P_4)$, with $P_4 \in D_{n-k}$. 
\end{enumerate}

All in all, this decomposition tells us that for any $P \in D_{n+1}$,
\begin{equation}\begin{split}
\area(P) = \lr{k + \frac{1}{2}} + \area(P_2) + \area(P_4)
\end{split}\end{equation}

This allows us to consider the area of a single $(n+1,n+1)$-Dyck path as the sum of the area of two smaller Dyck paths with shapes $(k,k)$ and $(n-k,n-k)$, plus some small factor, giving us the beginnings of a recurrence relation for the total area. A full example of how this can be used to find the total area for $(n,n)$-Dyck paths can be found in \cite{merlini}.

\section{The $(sn,n)$ Touchpoint Decomposition}

\subsection{Extending the Touchpoint Decomposition}

In the $(sn,n)$ case, we are unable to directly reuse the first touchpoint decomposition technique from the $(n,n)$ case. This is because the segments we decompose into might not be $(sn,n)$-Dyck paths.

\begin{example}
Consider the following $(sn,n)$-Dyck path with $s=2$ and $n=2$. 

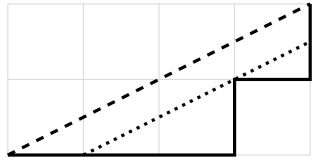
\begin{center}
\begin{tikzpicture}
    \draw[help lines, color=gray!30] (0,0) grid (4,2);
    
    \draw[very thick, color=black] (0,0) -- (1,0) -- (2,0) -- (3,0) -- (3,1) -- (4,1) -- (4,2);
    \draw[dashed, very thick, color=black] (0,0) -- (4,2);
    \draw[dotted, very thick, color=black] (1,0) -- (4,1.5);
\end{tikzpicture}
\end{center}

We can note that there are no touchpoints on the center line (the dashed line). This guarantees that the area between the dashed line and the dotted line will be part of the area of the path. However, the rest is a Dyck path of shape $(3,1.5)$; while this is the same $2:1$ ratio as our original Dyck path of shape $(4,2)$, it is not one that we can work with.
\end{example}

In order to resolve this issue, we will use a technique from \cite{krattenthaler2015lattice}, which was used to find the total number of $(sn,n)$-Dyck paths. Instead of just using the center line as in the $ (n,n) $ case, we define $ s $ lines
\begin{equation}\begin{split}
  L_{i} := \{(x,y):x = sy+i\},\quad 0\leq i <s
\end{split}\end{equation}
and use these lines to perform a touchpoint decomposition. We will characterize our path by its \textit{last} touchpoint on each of these lines. Note that we could have done a last touchpoint decomposition in the $(n,n)$ case, with the exact same results. This would require treating $(0,0)$ as a touchpoint and $(n,n)$ as not a touchpoint, which is the opposite of what we did.

To perform this decomposition, fix a path $ P\in D_{n+1}^{s} $ and consider the following procedure. The path $ P $ will meet the line $ L_{0} $ somewhere for the last time (excluding $ (s(n+1),n+1) $). Denote this point $ Z_{0} = (sy_{0},y_{0}) $. Let $ P_{0} $ be the subpath of $ P $ from $ (0,0) $ to $ Z_{0} $. The next point in $ P $ is forced to be $ Z_{0} +(1,0) \in L_{1}$. Then, for $ 1\leq k <s $, 
\begin{itemize}
  \item $ P $ meets $ L_{k} $ for the last time at $ Z_{k} = (sy_{k}+k, y_{k}) $,
  \item $ P_{k} $ is the subpath of $ P $ from $ Z_{k-1}+(1,0)\in L_{k} $ to $ Z_{k} $,
  \item the next point in $ P $ after $ P_{k} $ is forced to be $ Z_{k}+(1,0)\in L_{k+1} $.
\end{itemize}
Finally, let $ P_{s} $ be the subpath of $ P $ from $ Z_{s-1} +(1,0)$ to $ Z_{s}:=(s(n+1), n) $ and note that the last move in the path is to $ (s(n+1), n+1) $. This modified decomposition is shown in Figure \ref{fig:extended_touchpoint_decomposition}.

\begin{figure}[H]
\centering
\begin{tikzpicture}[scale=0.50]
  \def\n{12}
  \def\s{1.5}
  \def\TickSize{3.5pt}
  \draw[gray, very thin] (0,0) -- (0,\n+1) -- (\s*\n+\s,\n+1) --(\s*\n+\s, 0) -- (0,0);
  
  \draw[red, thick] (0,0) -- (\s*\n+\s,\n+1) node[anchor=south west] [xshift=-1.5cm,yshift=-0.1cm] {$x = s y$};
  
\coordinate (O) at (0,0);

  \coordinate (A) at (\s*2,2); 
  \coordinate (B) at (\s*3+1,3); 
  \coordinate (C) at (\s*5+2,5); 

  \coordinate (X) at (\s*2+1,2);
  \coordinate (Y) at (\s*3+2,3);

  \node[above left] at (A) {$Z_{0}$};
  \node[above] at (B) {$Z_{1}$};
  \node[above] at (C) {$Z_{2}$};
  \node[left] at (0,2) {$ y_{0} $};
  \node[left] at (0,3) {$ y_{1} $};
  \node[left] at (0,5) {$ y_{2} $};
  \node[left] at (0,\n-3) {$ y_{s-1} $};
  \node[left] at (0,\n-1) {$ n $};
  
  \foreach \y in {2,3,5, \n-3, \n-1, \n} {%
      \draw ($(0,\y) + (-\TickSize,0)$) -- ($(0,\y) + (\TickSize,0)$)
          node [right] {};
  }
  \foreach \point in {A, B, C}
      \fill [black] (\point) circle (3.25pt);

  \tikzset{decoration={snake,amplitude=.4mm,segment length=2mm,
                       post length=0mm,pre length=0mm}}
  \draw[thick, blue, decorate] (O) to[bend right=30] node[midway,below right, black] {$ P_{0} $} (A);
  \draw[thick, blue, decorate] (X) to[bend right=40] node[midway,below right, black] {$ P_{1} $} (B);
  \draw[thick, blue, decorate] (Y) to[bend right=30] node[midway,below right, black] {$ P_{2} $} (C);
  \path (\s*5+3,6) -- node[auto=false]{\ldots} (\s*5+5,6);
  \draw [very thick, -Stealth] (A) -- (X);
  \draw [very thick,dashed] (X) -- (B);
  \draw [very thick, -Stealth] (B) -- (Y);
  \draw [very thick,dashed] (Y)-- (C);
  \coordinate (J) at (\s*\n+\s, \n+1);
  \coordinate (K) at (\s*\n+\s, \n-1);
  \coordinate (L) at (\s*\n-\s, \n-3);
  \coordinate (M) at (\s*\n-\s-1, \n-3);
  \coordinate (N) at (\s*\n -\s-1 -1.5*\s, \n-4.5);
  \foreach \point in {J, K,  M}
      \fill [black] (\point) circle (2.5pt);

  \node[above ] at (M) {$ Z_{s-1} $};

  \draw [very thick, dashed] (L) -- (K);
  \draw [very thick, -Stealth] (K) -- (J);
  \draw [very thick, -Stealth] (M) -- (L);
  \draw [very thick, dashed] (N) -- (M);

  \tikzset{decoration={snake,amplitude=.4mm,segment length=2mm,
                       post length=0mm,pre length=0mm}}
  \draw[thick, blue, decorate] (\s*\n -\s-2 -1.5*\s+1, \n-5) to[bend right=30] node[midway,below] {} (M);
  \draw[thick, blue, decorate] (L) to[bend right=40] node[midway,below] {} (K);

\end{tikzpicture}
\caption{Last touchpoint decomposition with $s$ lines}
\label{fig:extended_touchpoint_decomposition}
\end{figure}
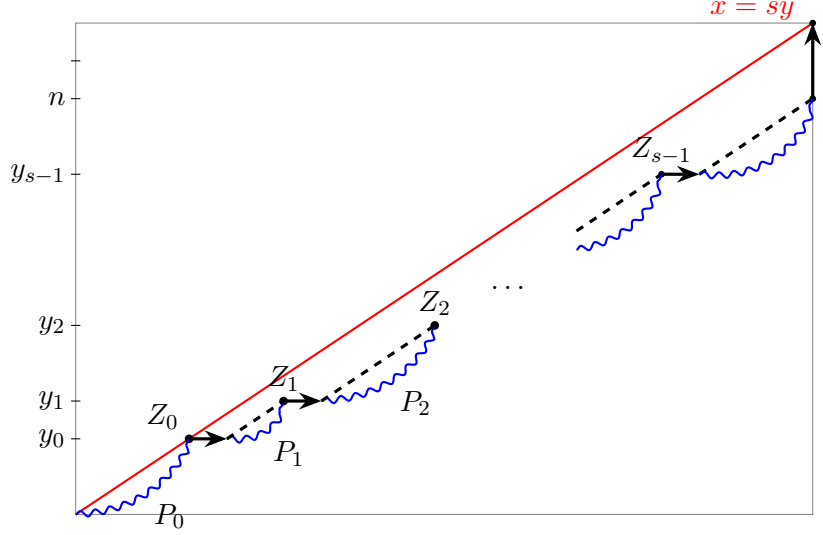

\begin{definition}[Path Type]
  The vector $ \vec{y} = (y_{0},\ldots, y_{s-1}) $ is called the \emph{type} of the path $ P $.
\end{definition}

Given a type $\vec{y} = (y_{0}, \ldots, y_{s-1})$, with $0 \leq y_{0} \leq \cdots \leq y_{s-1} \leq n$ and the convention $y_{s} = n$, we define its \emph{distances} $d_{0} = y_{0}$ and $d_{l} = y_{l} - y_{l-1}$ for $1 \leq l \leq s$. Note that $d_{0} + \cdots + d_{s} = n$, and that a type is uniquely determined by its distances. The value of this decomposition comes from the following counting lemma, which is implicit in the second proof of Theorem 10.4.5 in \cite{krattenthaler2015lattice}.

\begin{lemma}[Paths of a Given Type]\label{lem:type_count}
Let $\vec{y}$ be a type with distances $d_{0}, \ldots, d_{s}$. Then the map $P \mapsto (P_{0}, \ldots, P_{s})$ is a bijection between the set of paths $P \in D^s_{n+1}$ of type $\vec{y}$ and the product $D^s_{d_0} \times \cdots \times D^s_{d_s}$, where each subpath $P_{l}$ is identified with an $(sd_l, d_l)$-Dyck path by translation. In particular, the number of paths of type $\vec{y}$ is
\begin{equation}\begin{split}
    F^s_{d_0} F^s_{d_1} \cdots F^s_{d_s}
\end{split}\end{equation}
\end{lemma}
\begin{proof}
The subpath $P_{0}$ runs from $(0,0)$ to $(sy_{0}, y_{0})$ while staying weakly beneath $L_{0}$, so it is itself an $(sd_0, d_0)$-Dyck path. For $1 \leq l \leq s$, after its last touchpoint on $L_{l-1}$ the path lies strictly beneath $L_{l-1}$, and hence (since it moves between lattice points) weakly beneath $L_{l}$. Consequently, $P_{l}$ runs from $Z_{l-1} + (1,0)$ to $Z_{l}$, both points of $L_{l}$, while staying weakly beneath $L_{l}$; translating its starting point to the origin identifies $P_{l}$ with an $(sd_l, d_l)$-Dyck path.

Conversely, concatenating any tuple of such paths, with the forced step $(1,0)$ inserted between consecutive subpaths and the final step $(0,1)$ appended, produces a path in $D^s_{n+1}$ of type $\vec{y}$, and the two constructions are mutually inverse. Since the subpaths may be chosen independently, the product formula follows.
\end{proof}

Our derivations will be phrased in terms of ordinary generating functions. We write
\begin{equation}\begin{split}
    F(z) := \sum_{n=0}^{\infty} F^s_n z^n
\end{split}\end{equation}
for the generating function of the Fuss-Catalan numbers; the parameter $s$ is fixed throughout our arguments, so we suppress it from this notation. For a power series $H(z)$, we write $[H(z)]_n$ for the coefficient of $z^n$ in $H(z)$, so that, for example, $F^s_n = [F(z)]_n$ and $n F^s_n = [zF'(z)]_n$.

Summing the product formula over all types recovers the enumeration $F^s_{n+1} = [F(z)^{s+1}]_{n}$ from \cite{krattenthaler2015lattice}. Both this identity and Lemma \ref{lem:type_count} will be used repeatedly in the computation below.

\subsection{A Recursive Formula for the Total Area}
We are interested in studying the total area of all $ (sn,n) $-Dyck paths, i.e., the quantity
\begin{equation}\begin{split}
  A^s_n = \sum_{P\in D_{n}^{s}}\area(P)
\end{split}\end{equation}

We write $A(z) := \sum_{n=0}^{\infty} A^s_n z^n$ for the generating function of the total areas.

In this section, we derive a recursive formula for $ A^s_{n+1} $ using the touchpoint decomposition above.
Note that the data for the decomposition of a path is entirely encoded in its type vector $ \vec{y} $. Hence, in order to iterate over all possible paths, we will first iterate over all possible paths of a given type and then iterate over all possible types $ 0\leq y_{0}\leq y_{1}\leq \cdots \leq y_{s-1} \leq n $, which we will find convenient to index by their distance vectors. For simplicity, set $ y_{s}=n $.

To begin, fix a type $ \vec{y} = (y_{0},\ldots, y_{s-1}) $ and a path $ P\in D_{n+1}^{s}$ with type $\vec{y}$. The area of $ P $ decomposes into the area contributions from the portions corresponding to the subpaths $ P_{l} $ for $ 0\leq l \leq s $. The contribution of $ P_{0} $ to the total area is precisely $ \area(P_{0}) $. \\

\textbf{\underline{Contribution of $0< l < s$}}: The parallelogram $ I_{l} $ in Figure \ref{fig:main_contributions} contributes $ l(y_{l}-y_{l-1}) $, while $ P_{l} $ contributes $ \area(P_{l}) $. \\

\begin{figure}[H]
    \centering
    \begin{tikzpicture}[scale=0.80, every node/.style={scale=1}]
      \def\n{4}
      \def\s{2}
      \def\TickSize{3pt}
      \draw[gray] (-1,0) -- (-1,\n+1);
      
      \draw[red, thick] (0,0) -- (\s*\n+1*\s,\n+1) node[anchor=south west] [xshift=-1.5cm,yshift=-0.1cm] {$x = s y$};
      
      \coordinate (O) at (0,0);
    
      \coordinate (X) at (\s*2+2,2);
      \coordinate (Y) at (\s*3+4,3);

      \coordinate (J) at (\s*\n+\s, \n+1);
      \coordinate (K) at (\s*\n+\s, \n-1);
      \coordinate (b) at (\s*\n+\s-4, \n-1);
      \coordinate (L) at (\s*\n-\s, \n-3);
      \coordinate (M) at (\s*\n-\s-2, \n-3);
      \coordinate (a) at (\s*\n-\s-4, \n-3);
      \coordinate (N) at (\s*\n -\s-2 -1.5*\s, \n-4.5);
      \foreach \point in { K,  M, a, b}
          \fill [black] (\point) circle (2.5pt);
    
      \definecolor{darkgreen}{RGB}{0,100,0}
      \coordinate (I) at (\s*\n+\s-4,\n-2);
      \node at (I) [darkgreen] {$ I_{l} $};
    
      \coordinate (P) at (\s*\n+\s-1, \n-2.5);
      \node[below right] at (P) [blue] {$ P_{l} $};
    
      \node[below] at (M) {$ Z_{l-1} $};
      \node[above] at (K) {$ Z_{l}=(sy_{l}+l,y_{l}) $};
      \node[above left] at (b) {$ (sy_{l},y_{l}) $};
      \node[above left] at (a) {$ (sy_{l-1},y_{l-1}) $};
      
      \draw [very thick, dashed] (L) -- (K);
      \draw [ dashed] (a) -- (M);
      \draw [ dashed] (b) -- (K);
      \draw [very thick, -Stealth] (M) -- (L);
      
      \foreach \y in {\n-3,\n-1} {%
          \draw ($(-1,\y) + (-\TickSize,0)$) -- ($(-1,\y) + (\TickSize,0)$)
              node [right] {};
      }
      \node[left] at (-1,\n-3) {$ y_{l-1} $};
      \node[left] at (-1,\n-1) {$ y_{l} $};
      \tikzset{decoration={snake,amplitude=.4mm,segment length=2mm,
                           post length=0mm,pre length=0mm}}
      \draw[thick, blue, decorate] (L) to[bend right=40] node[midway,below] {} (K);

    \end{tikzpicture}
    \caption{Contribution of $0 < l < s$}
    \label{fig:main_contributions}
\end{figure}
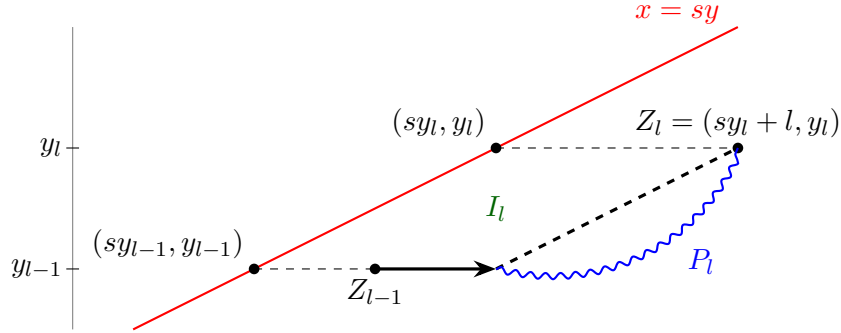

\textbf{\underline{Contribution of $l=s$}}: Again, the parallelogram $ I_{s} $ in Figure \ref{fig:last_contribution} contributes $ s(y_{s}-y_{s-1}) = s(n-y_{s-1}) $ and $ P_{s} $ contributes $ \area(P_{s}) $. Additionally, the triangle $ T $ contributes an area of $ \frac{s}{2} $. \\

\begin{figure}[H]
    \centering
    \begin{tikzpicture}[scale=0.80, every node/.style={scale=1}]
      \def\n{4}
      \def\s{2}
      \def\TickSize{3pt}
    
      \draw[red, thick] (0,0) -- (\s*\n+\s,\n+1) node[anchor=south west] [xshift=-1.5cm,yshift=-0.1cm] {$x = s y$};
      
      \coordinate (O) at (0,0);
    
      \coordinate (X) at (\s*2+2,2);
      \coordinate (Y) at (\s*3+4,3);

      \coordinate (J) at (\s*\n+\s, \n+1);
      \coordinate (K) at (\s*\n+\s, \n-1);
      \coordinate (b) at (\s*\n+\s-4, \n-1);
      \coordinate (L) at (\s*\n-\s, \n-3);
      \coordinate (M) at (\s*\n-\s-2, \n-3);
      \coordinate (a) at (\s*\n-\s-4, \n-3);
      \coordinate (N) at (\s*\n -\s-2 -1.5*\s, \n-4.5);
      \foreach \point in { K,  M, a, b}
          \fill [black] (\point) circle (2.5pt);
    
      \definecolor{darkgreen}{RGB}{0,100,0}
      \coordinate (I) at (\s*\n+\s-4,\n-2);
      \node at (I) [darkgreen] {$ I_{s} $};
    
      \coordinate (P) at (\s*\n+\s-1, \n-2.5);
      \node[below right] at (P) [blue] {$ P_{s} $};
      
      \coordinate (T) at (\s*\n+\s-1, \n);
      \node[below left] at (T) [darkgreen] {$ T $};
    
      \node[below] at (M) {$ Z_{s-1} $};
      \node[right] at (K) {$ Z_{s} $};
      \node[above left] at (b) {$ (sy_{s},y_{s}) $};
      \node[above left] at (a) {$ (sy_{s-1},y_{s-1}) $};
      
      \draw [very thick, dashed] (L) -- (K);
      \draw [ dashed] (a) -- (M);
      \draw [ dashed] (b) -- (K);
      \draw [very thick, -Stealth] (K) -- (J);
      \draw [very thick, -Stealth] (M) -- (L);
      
      \tikzset{decoration={snake,amplitude=.4mm,segment length=2mm,
                           post length=0mm,pre length=0mm}}
      \draw[thick, blue, decorate] (L) to[bend right=40] node[midway,below] {} (K);

    \end{tikzpicture}
    \caption{Contribution of $l = s$}
    \label{fig:last_contribution}
\end{figure}
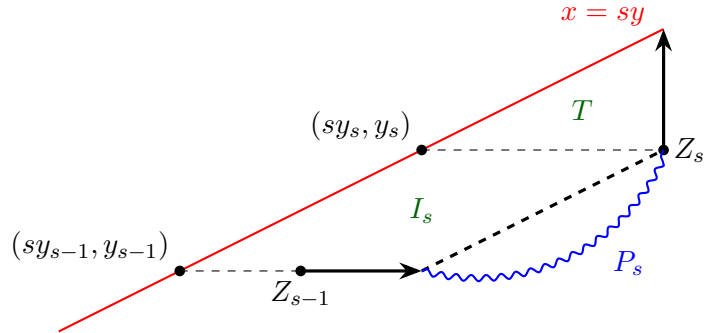

Combining these contributions, and rewriting in terms of the distances $d_{l}$ of the type of $P$, gives
\begin{equation}\begin{split}
  \area(P) = \frac{s}{2} + \sum_{l=1}^{s} l \cdot d_{l} + \sum_{l=0}^{s}\area(P_{l})
\end{split}\end{equation}

We can then split the total area by type. Since a type is uniquely determined by its distances, we may specify each type by its distance vector $\vec{d} = (d_{0}, \ldots, d_{s})$, so that summing over all types is equivalent to summing over $d_{0} + \cdots + d_{s} = n$ with $d_{0}, \ldots, d_{s} \geq 0$. This gives
\begin{equation}\begin{split}
  A^s_{n+1} = \sum_{P\in D_{n+1}^{s}} \area(P) = \sum_{d_{0}+\cdots+d_{s}=n} \sum_{\substack{P\in D_{n+1}^{s}\\\text{of type }\vec{d}}} \area(P) = \\
  \sum_{d_{0}+\cdots+d_{s}=n} \sum_{\substack{P\in D_{n+1}^{s}\\\text{of type }\vec{d}}} \lr{\frac{s}{2} + \sum_{l=1}^{s} l \cdot d_{l} + \sum_{l=0}^{s}\area(P_{l})}
\end{split}\end{equation}

This gives us three sums to evaluate, one for each of the three terms in our $\area(P)$ formula. Since $ \frac{s}{2} $ is constant, its corresponding sum evaluates to
\begin{equation}\begin{split}
  \sum_{P\in D_{n+1}^{s}}\frac{s}{2} = \frac{s}{2}F^s_{n+1} = \frac{s}{2}[F(z)^{s+1}]_{n} 
\end{split}\end{equation}

The remaining two sums both distribute a weight over a single coordinate of the distance vector, and we will evaluate both with the following identity.

\begin{lemma}[Weighted Type Sums]\label{lem:weighted_type_sum}
Let $(c_d)_{d \geq 0}$ be a sequence with generating function $C(z) = \sum_{d=0}^{\infty} c_d z^d$. Then, for each $0 \leq l \leq s$,
\begin{equation}\begin{split}
    \sum_{d_{0}+\cdots+d_{s}=n} \lr{\prod_{\substack{0 \leq i \leq s\\i \neq l}} F^s_{d_i}} c_{d_l} = [C(z) F(z)^{s}]_{n}
\end{split}\end{equation}
\end{lemma}
\begin{proof}
Swapping the coordinates $d_l$ and $d_s$ permutes the terms of the sum, so we may assume $l = s$. The sum then factors as the convolution
\begin{equation}\begin{split}
    \sum_{0 \leq d_s \leq n} c_{d_s} \sum_{d_{0}+\cdots+d_{s-1}=n-d_s} \prod_{0 \leq i \leq s-1} F^s_{d_i} = \sum_{0 \leq d_s \leq n} c_{d_s} \cdot [F(z)^{s}]_{n-d_s} = [C(z) F(z)^{s}]_{n}
\end{split}\end{equation}
\end{proof}

Addressing the sum from the parallelogram contributions, we move the sum over $l$ outward, apply Lemma \ref{lem:type_count} to count the paths of each type, and then apply Lemma \ref{lem:weighted_type_sum} with $c_d = d F^s_d$, so that $C(z) = zF'(z)$:
\begin{equation}\begin{split}
  \sum_{d_{0}+\cdots+d_{s}=n} \sum_{\substack{P\in D_{n+1}^{s}\\\text{of type }\vec{d}}} \sum_{l=1}^{s} l \cdot d_{l} = \sum_{l=1}^{s}l\sum_{d_{0}+\cdots+d_{s}=n} \lr{\prod_{\substack{0 \leq i \leq s\\i \neq l}} F^s_{d_i}} d_l F^s_{d_l} = \\
  \lr{\sum_{l=1}^{s}l} [zF'(z)F(z)^{s}]_{n} = \frac{s(s+1)}{2}[zF^{\prime}(z)F(z)^{s}]_{n}
\end{split}\end{equation}

For the recursive area sum, we first condition on the subpath $P_l = \gamma$ (here Lemma \ref{lem:type_count} counts the paths of type $\vec{d}$ with a prescribed subpath $P_{l} = \gamma$), then sum the areas over $\gamma \in D_{d_l}^{s}$ and apply Lemma \ref{lem:weighted_type_sum} with $c_d = A^s_d$:
\begin{equation}\begin{split}
    \sum_{l=0}^{s}\sum_{d_{0}+\cdots+d_{s}=n} \sum_{\substack{P\in D_{n+1}^{s}\\\text{of type }\vec{d}}} \area(P_{l}) = \sum_{l=0}^{s}\sum_{d_{0}+\cdots+d_{s}=n}\sum_{\gamma\in D_{d_l}^{s}} \area(\gamma) \prod_{\substack{0\leq i \leq s\\i\neq l}} F^s_{d_i} = \\
    \sum_{l=0}^{s}\sum_{d_{0}+\cdots+d_{s}=n} \lr{\prod_{\substack{0\leq i \leq s\\i\neq l}} F^s_{d_i}} A^s_{d_l} = \sum_{l=0}^{s} [A(z)F(z)^{s}]_{n} = (s+1)[A(z)F(z)^{s}]_{n}
\end{split}\end{equation}

Combining these sums establishes our desired recursive total area formula.
\begin{proposition}[Recursive Area Formula]\label{prop:recursive}
For all $n \geq 0$ and $s \geq 1$,
\begin{equation}\label{eq:recursive}
  A^s_{n+1} =\frac{s}{2}[F(z)^{s+1}]_{n} + \frac{s(s+1)}{2}[zF^{\prime}(z)F(z)^{s}]_{n}+ (s+1)[A(z)F(z)^{s}]_{n}
\end{equation}
\end{proposition}

\section{Total Area Formula}\label{sec:total_area_formula}

\subsection{The Area Generating Function}

Using the recursive total area formula in Equation \ref{eq:recursive}, we can simplify by expressing the formula as a functional equation and then applying various transformations. Multiplying Equation \ref{eq:recursive} by $z^n$ and summing over $n \geq 0$ (noting that $A^s_0 = 0$) gives
\begin{equation}\begin{split}
    \frac{A(z)}{z} =  \frac{s}{2}F(z)^{s+1} + \frac{s(s+1)}{2}zF^{\prime}(z)F(z)^{s} + (s+1)A(z)F(z)^{s}
\end{split}\end{equation}

Collecting the $A(z)$ terms on the left, factoring $\frac{s}{2} z F(z)^{s+1}$ out of the right-hand side, and solving for $A(z)$ yields
\begin{equation}\begin{split}\label{eq:area_isolation}
    A(z) = \frac{\frac{s}{2} z F(z)^{s+1} \lr{1 + (s+1) z \frac{F'(z)}{F(z)}}}{1-(s+1)zF(z)^{s}}
\end{split}\end{equation}

To continue simplifying, we will rely on a useful identity for $F(z)$, referenced in \cite{mlotkowski2012densitiesraneydistributions} (Equation 6). This identity's derivative, which we have placed just below, will also prove useful.
\begin{equation}\begin{gathered}\label{eq:fz_identities}
    F(z) = 1 + z F(z)^{s+1} \\
    F'(z) = F(z)^{s+1} + (s+1) z F(z)^{s} F'(z)
\end{gathered}\end{equation}

Two immediate consequences of these identities are
\begin{align}\label{eq:fz_consequences}
    \frac{1}{F(z)} = 1 - z F(z)^{s} && \text{and} && 1 - (s+1) z F(z)^{s} = \frac{F(z)^{s+1}}{F'(z)}
\end{align}
each obtained by rearranging the corresponding identity in Equation \ref{eq:fz_identities}.

Applying the identity $1 - (s+1) z F(z)^{s} = F(z)^{s+1} / F'(z)$ to Equation \ref{eq:area_isolation} and simplifying yields
\begin{equation}\begin{split}\label{eq:functional_equation}
    A(z) = \frac{\frac{s}{2} z F(z)^{s+1} \lr{1 + (s+1) z \frac{F'(z)}{F(z)}}}{\frac{F(z)^{s+1}}{F'(z)}} = \frac{s}{2} zF'(z) \lr{1 + (s+1) \frac{zF'(z)}{F(z)}}
\end{split}\end{equation}

\begin{remark}
If we directly plug in the formulas for $zF'(z)$ and $\frac{1}{F(z)}$ into Equation \ref{eq:functional_equation} and simplify, we end up with a lengthy triple-sum formula for $A^s_n$, closely resembling one derived by Mansour and Ramírez for a related problem (\cite{mansour2021enumerations}, Theorem 3.3). Fortunately, the transformations below will allow us to do substantially better.
\end{remark}

Substituting $\frac{1}{F(z)} = 1 - z F(z)^{s}$ into the remaining factor of Equation \ref{eq:functional_equation} gives
\begin{equation}\begin{split}
    (s+1) \frac{zF'(z)}{F(z)} = (s+1) zF'(z) \lr{1 - z F(z)^{s}} = \\
    (s+1) zF'(z) - (s+1) z^2 F'(z) F(z)^{s} = \\
    szF'(z) + z \lr{F'(z) - (s+1) z F(z)^{s} F'(z)} = \\
    szF'(z) + z F(z)^{s+1} = \\
    szF'(z) + F(z) - 1
\end{split}\end{equation}
where the final two equalities apply the derivative form and then the original identity from Equation \ref{eq:fz_identities}. Hence $1 + (s+1) \frac{zF'(z)}{F(z)} = F(z) + szF'(z)$, and Equation \ref{eq:functional_equation} simplifies to
\begin{equation}\begin{split}\label{eq:functional_transformed1}
    A(z) = \frac{1}{2} szF'(z) \cdot \lr{F(z) + szF'(z)}
\end{split}\end{equation}

Compared to Equation \ref{eq:functional_equation}, this formula will be much more pleasant to work with.

\begin{remark}
Banderier and Gittenberger use the kernel method to find generating functions for the total area below directed lattice paths (\cite{banderier2006area}, Theorem 5), and Section 6 of their paper shows how paths lying below a line of rational slope, such as $(sn,n)$-Dyck paths, can be recast as directed lattice paths. Combining these two results gives an alternative derivation of Equation \ref{eq:functional_transformed1}.
\end{remark}

In fact, the formula simplifies further. Let $G(z) := \sum_{n=0}^{\infty} {(s+1)n \choose n} z^n$. Since $(sn+1) F^s_n = {(s+1)n \choose n}$, we have
\begin{equation}\begin{split}\label{eq:functional_simplification}
    F(z) + szF'(z) = \sum_{n=0}^{\infty} (1 + sn) F^s_n z^n = G(z)
\end{split}\end{equation}
and hence also $szF'(z) = G(z) - F(z)$. Equation \ref{eq:functional_transformed1} therefore becomes
\begin{equation}\begin{split}\label{eq:functional_transformed2}
    A(z) = \frac{1}{2} \lr{G(z) - F(z)} G(z) = \frac{1}{2} \lr{G(z)^2 - F(z) G(z)}
\end{split}\end{equation}

We can now proceed to our derivation of the area formula.

\subsection{Extracting the Exact Formula}

We will make repeated use of the following instance of the binomial absorption identity (see \cite{knuth1989concrete}, Identity 5.7)
\begin{equation}\begin{split}\label{eq:absorption}
    {(s+1)m+1 \choose m} = \frac{(s+1)m+1}{sm+1} {(s+1)m \choose m}
\end{split}\end{equation}

\begin{theorem}[Exact Total Area Formula]\label{thm:exact_area} For all $n \geq 0, s \geq 1$,
    \begin{equation}\begin{split}\label{eq:total_area}
        A^s_n = \frac{1}{2} \lr{\sum_{k=0}^{n} {(s+1)k \choose k} {(s+1)(n-k) \choose n-k} - {(s+1)n+1 \choose n}}
    \end{split}\end{equation}
\end{theorem}
\begin{proof}
We extract the coefficient of $z^n$ from each term of Equation \ref{eq:functional_transformed2}. The first term gives the sum in Equation \ref{eq:total_area} directly:
\begin{equation}\begin{split}
    [G(z)^2]_n = \sum_{k=0}^{n} {(s+1)k \choose k} {(s+1)(n-k) \choose n-k}
\end{split}\end{equation}
For the second term, Equation \ref{eq:absorption} rewrites $F^s_k = \frac{1}{sk+1} {(s+1)k \choose k} = \frac{1}{(s+1)k+1} {(s+1)k+1 \choose k}$, which puts the coefficient in the form of convolution Identity 5.62 from \cite{knuth1989concrete}:
\begin{equation}\begin{split}
    [F(z) G(z)]_n = \sum_{k=0}^{n} F^s_k {(s+1)(n-k) \choose n-k} = {(s+1)n+1 \choose n}
\end{split}\end{equation}

Subtracting these two coefficients and halving yields Equation \ref{eq:total_area}, completing the proof.
\end{proof}

\begin{remark}
In the specific case $s=1$, this becomes the familiar formula for the area of Dyck paths,
\begin{equation}\begin{split}
A^s_n = \frac{1}{2} \lr{\sum_{k=0}^{n} {2k \choose k} {2(n-k) \choose n-k} - {2n+1 \choose n}} = \frac{1}{2} \lr{4^n - {2n+1 \choose n}}
\end{split}\end{equation}
\end{remark}

\section{Area Bounds and Asymptotics in $n$}

\subsection{Binomial Estimates}

While the area formula we have derived is useful, in many cases we would like to have bounds or asymptotics that are easier to work with. We will begin by finding explicit upper and lower bounds for the expression
\begin{equation}\begin{split}\label{eq:inner_sum}
    \sum_{k=0}^{n} {(s+1)k \choose k} {(s+1)(n-k) \choose n-k}
\end{split}\end{equation}
and will then use these bounds to find bounds on $A^s_n$. This expression is also interesting in its own right, and appears in the mathematical literature for other problems. In this section, we will focus on finding bounds that are good for large $n$, though in the next section we will also examine bounds that are good for large $s$.

One particularly useful pair of bounds, published by St\u{a}nic\u{a} (\cite{stanica2001good}, Theorem 2.6), is

\begin{lemma}[St\u{a}nic\u{a}'s Binomial Bounds]\label{lem:stanica_bounds} Let $s, n \in \N$, with $s \geq 1$ and $n \geq 1$. Then, 
\begin{equation}\begin{split}
    \frac{1}{\sqrt{2\pi}} \sqrt{\frac{s+1}{s}} \cdot \frac{1}{\sqrt{n}} \lr{\frac{(s+1)^{s+1}}{s^s}}^n \cdot e^{-\frac{1}{8n}} < {(s+1)n \choose n} < \frac{1}{\sqrt{2\pi}} \sqrt{\frac{s+1}{s}} \cdot \frac{1}{\sqrt{n}} \lr{\frac{(s+1)^{s+1}}{s^s}}^n
\end{split}\end{equation}
\end{lemma}

Since Lemma \ref{lem:stanica_bounds} requires a positive index, we separate out the $k = 0$ and $k = n$ terms of Equation \ref{eq:inner_sum}, each of which equals ${(s+1)n \choose n}$, rewriting the sum as
\begin{equation}\begin{split}\label{eq:inner_sum_rewritten}
    \sum_{k=1}^{n-1} {(s+1)k \choose k} {(s+1)(n-k) \choose n-k} + 2 {(s+1)n \choose n}
\end{split}\end{equation}
so that the lemma applies to every term of the remaining summation. The bounds we derive below will require $n \geq 2$, and so will not cover the cases $n = 0$ and $n = 1$.

\subsection{Upper and Lower Bounds on the Sum}

\begin{lemma}[Upper Bound on the Sum]\label{lem:inner_sum_upper_bound} Let $s, n \in \N$, with $s \geq 1$ and $n \geq 2$. Then, 
\begin{equation}\begin{split}
    \sum_{k=0}^{n} {(s+1)k \choose k} {(s+1)(n-k) \choose n-k} < \frac{1}{2 \pi} \frac{s+1}{s} \lr{\frac{(s+1)^{s+1}}{s^s}}^n \lr{\pi - \frac{2}{\sqrt{n}}} + 2 {(s+1)n \choose n}
\end{split}\end{equation}
\end{lemma}
\begin{proof}
Applying Lemma \ref{lem:stanica_bounds} to Equation \ref{eq:inner_sum_rewritten} yields
\begin{equation}\begin{split}\label{eq:initial_upper_bound}
    \frac{1}{2 \pi} \frac{s+1}{s} \lr{\frac{(s+1)^{s+1}}{s^s}}^n \cdot \sum_{k=1}^{n-1} \frac{1}{\sqrt{k(n-k)}} + 2 {(s+1)n \choose n}
\end{split}\end{equation}

We now bound the remaining sum using convexity. Note that for any $p, q > 0$, the function $k^{-p} (n-k)^{-q}$ is convex on $(0,n)$: its logarithm $-p \ln{k} - q \ln{(n-k)}$ has second derivative $\frac{p}{k^2} + \frac{q}{(n-k)^2} > 0$, and exponentiating preserves convexity by the composition rule in \cite[Section 3.2.4]{boyd2004convex}. In particular, the summand $f(k) = \frac{1}{\sqrt{k(n-k)}}$ is convex in $k$ on $(0,n)$. This lets us apply the Hermite-Hadamard inequality (see \cite{niculescu2006convex}), which states that any function $g$ that is convex on $[a,b]$ satisfies
\begin{equation}\begin{split}\label{eq:hermite_hadamard}
    g\lr{\frac{a+b}{2}} \leq \frac{1}{b-a} \int_{a}^{b} g(x) dx \leq \frac{g(a) + g(b)}{2}
\end{split}\end{equation}

Applying the left half of Equation \ref{eq:hermite_hadamard} to $f$ on each interval $\left[k - \frac{1}{2}, k + \frac{1}{2}\right]$, and summing over $1 \leq k \leq n-1$, yields
\begin{equation}\begin{split}
    \sum_{k=1}^{n-1} \frac{1}{\sqrt{k(n-k)}} \leq \int_{\frac{1}{2}}^{n-\frac{1}{2}} \frac{1}{\sqrt{k(n-k)}} dk
\end{split}\end{equation}

For any $0 < t \leq \frac{n}{2}$, this integral evaluates to
\begin{equation}\begin{split}\label{eq:arcsin_integral}
    \int_{t}^{n-t} \frac{1}{\sqrt{k(n-k)}} dk = 2\arcsin\lr{\sqrt{1 - \frac{t}{n}}} - 2\arcsin\lr{\sqrt{\frac{t}{n}}} = \pi - 4\arcsin\lr{\sqrt{\frac{t}{n}}}
\end{split}\end{equation}
where the second equality uses the identity $\arcsin\lr{\sqrt{x}} + \arcsin\lr{\sqrt{1-x}} = \frac{\pi}{2}$. Taking $t = \frac{1}{2}$ and applying the Taylor series lower bound $\arcsin(x) \geq x$ gives
\begin{equation}\begin{split}
    \sum_{k=1}^{n-1} \frac{1}{\sqrt{k(n-k)}} \leq \pi - 4\arcsin\lr{\frac{1}{\sqrt{2n}}} \leq \pi - \frac{2\sqrt{2}}{\sqrt{n}} < \pi - \frac{2}{\sqrt{n}}
\end{split}\end{equation}

Plugging this back into Equation \ref{eq:initial_upper_bound} completes the proof.
\end{proof}

\begin{remark}
While we could also simplify the ${(s+1)n \choose n}$ term in Lemma \ref{lem:inner_sum_upper_bound}, we will leave it unsimplified for the time being. This is because we will be able to combine it with similar terms when applying the bound to the formula for the total area.
\end{remark}

\begin{lemma}[Lower Bound on the Sum]\label{lem:inner_sum_lower_bound} Let $s, n \in \N$, with $s \geq 1$ and $n \geq 2$. Then, 
\begin{equation}\begin{split}
    \sum_{k=0}^{n} {(s+1)k \choose k} {(s+1)(n-k) \choose n-k} > \frac{1}{2 \pi} \frac{s+1}{s} \lr{\frac{(s+1)^{s+1}}{s^s}}^n \lr{\pi - \frac{5}{\sqrt{n}}} + 2 {(s+1)n \choose n}
\end{split}\end{equation}
\end{lemma}
\begin{proof}
We proceed similarly to the proof for the upper bound, this time applying the lower bound from Lemma \ref{lem:stanica_bounds}. Combining the two resulting exponential factors via $\frac{1}{8k} + \frac{1}{8(n-k)} = \frac{n}{8k(n-k)}$, and then applying the standard bound $1 - x < e^{-x}$, yields the initial lower bound
\begin{equation}\begin{split}\label{eq:initial_lower_bound}
    \frac{1}{2 \pi} \frac{s+1}{s} \lr{\frac{(s+1)^{s+1}}{s^s}}^n \cdot \sum_{k=1}^{n-1} \lr{\frac{1}{\sqrt{k(n-k)}} - \frac{n}{8} \cdot \frac{1}{(k(n-k))^{\frac{3}{2}}}} + 2 {(s+1)n \choose n}
\end{split}\end{equation}

We will now find a lower bound for the summation in Equation \ref{eq:initial_lower_bound}, bounding its two sums separately. For the first sum, recall from the proof of Lemma \ref{lem:inner_sum_upper_bound} that $\frac{1}{\sqrt{k(n-k)}}$ is convex in $k$ on $(0,n)$. Applying the right half of the Hermite-Hadamard inequality (Equation \ref{eq:hermite_hadamard}) on each interval $[k, k+1]$, and summing over $1 \leq k \leq n-2$, gives
\begin{equation}\begin{split}
    \sum_{k=1}^{n-1} \frac{1}{\sqrt{k(n-k)}} \geq \int_{1}^{n-1} \frac{1}{\sqrt{k(n-k)}} dk + \frac{1}{\sqrt{n-1}} \geq \\
    \pi - 4\arcsin\lr{\frac{1}{\sqrt{n}}} + \frac{1}{\sqrt{n}} \geq \pi - \lr{\pi\sqrt{2} - 1} \frac{1}{\sqrt{n}}
\end{split}\end{equation}
where the integral was evaluated in Equation \ref{eq:arcsin_integral} (with $t = 1$). The final step uses the bound $\arcsin(x) \leq \frac{\pi \sqrt{2}}{4} x$ for $0 \leq x \leq \frac{1}{\sqrt{2}}$, which holds because $\arcsin$ is convex and so lies below its chord on $\left[0, \frac{1}{\sqrt{2}}\right]$.

For the second sum, the summand $\frac{1}{(k(n-k))^{\frac{3}{2}}}$ is again of the form $k^{-p} (n-k)^{-q}$, hence convex on $(0,n)$. Applying the left half of the Hermite-Hadamard inequality on each interval $\left[k - \frac{1}{2}, k + \frac{1}{2}\right]$, and summing over $1 \leq k \leq n-1$, gives
\begin{equation}\begin{split}
    \frac{n}{8} \sum_{k=1}^{n-1} \frac{1}{(k(n-k))^{\frac{3}{2}}} \leq \frac{n}{8} \int_{\frac{1}{2}}^{n-\frac{1}{2}} \frac{1}{(k(n-k))^{\frac{3}{2}}} dk = \frac{n-1}{n\sqrt{2n-1}} \leq \frac{1}{\sqrt{n}}
\end{split}\end{equation}
where the integral was evaluated using the antiderivative $\frac{2(2k-n)}{n^2 \sqrt{k(n-k)}}$, and the final step uses $2n-1 \geq n$.

Combining these two bounds yields
\begin{equation}\begin{split}
    \sum_{k=1}^{n-1} \lr{\frac{1}{\sqrt{k(n-k)}} - \frac{n}{8} \cdot \frac{1}{(k(n-k))^{\frac{3}{2}}}} \geq \pi - \pi\sqrt{2} \cdot \frac{1}{\sqrt{n}} > \pi - \frac{5}{\sqrt{n}}
\end{split}\end{equation}
since $\pi\sqrt{2} < 5$. Plugging this back into Equation \ref{eq:initial_lower_bound} completes the proof.
\end{proof}

It is worth noting that the bounds in Lemmas \ref{lem:inner_sum_lower_bound} and \ref{lem:inner_sum_upper_bound} converge asymptotically, giving
\begin{corollary}[Asymptotics of the Sum]\label{cor:inner_sum_asymp} Let $s,n \in \N$, with $s \geq 1$. Then, as $n$ grows large,
\begin{equation}\begin{split}
    \sum_{k=0}^{n} {(s+1)k \choose k} {(s+1)(n-k) \choose n-k} \sim \frac{1}{2} \frac{s+1}{s} \lr{\frac{(s+1)^{s+1}}{s^s}}^n
\end{split}\end{equation}
\end{corollary}

\begin{remark}\label{rem:chess_guy_conjecture}
This confirms a conjecture of Vaclav Kotesovec. In the formula section of the OEIS entry \cite{oeisA006256}, extending an observation of Benoit Cloitre for the case $m=3$, Kotesovec stated without proof (November 2012) that
\begin{equation*}
    \sum_{k=0}^{n} {mk \choose k} {m(n-k) \choose n-k} \sim \frac{1}{2} \frac{m}{m-1} \lr{\frac{m^m}{(m-1)^{m-1}}}^n
\end{equation*}
for general $m$. Setting $m = s+1$, this is precisely the statement of Corollary \ref{cor:inner_sum_asymp}.
\end{remark}

\begin{remark}\label{rem:merlini_tennis_balls}
One interesting appearance of this sum is in \cite{MERLINI2002307}, in which the authors discuss the $s$-tennis ball problem. In particular, in Theorem 5.2, the authors include this sum as part of the formula for one of their main results. Lemmas \ref{lem:inner_sum_lower_bound} and \ref{lem:inner_sum_upper_bound}, along with this corollary, can be used to find bounds and asymptotics for their formula.
\end{remark}

\subsection{Bounds and Asymptotics for the Area}

Now that we have good bounds on Equation \ref{eq:inner_sum}, we can apply them to some relevant problems. We will begin by applying them to our total area formula (Equation \ref{eq:total_area}).
\begin{theorem}[Bounds on Total Area]\label{thm:total_area_bounds} Let $s, n \in \N$, with $s \geq 1$ and $n \geq 2$. Then, 
\begin{equation}\begin{split}
    \frac{1}{4 \pi} \frac{s+1}{s} \lr{\frac{(s+1)^{s+1}}{s^s}}^n \lr{\pi - \frac{5}{\sqrt{n}}} < A^s_n < \frac{1}{4\pi} \frac{s+1}{s} \lr{\frac{(s+1)^{s+1}}{s^s}}^n \lr{ \pi + \frac{1}{\sqrt{n}} }
\end{split}\end{equation}
\end{theorem}
\begin{proof}
We can apply Lemmas \ref{lem:inner_sum_lower_bound} and \ref{lem:inner_sum_upper_bound} to the area equation, which gives
\begin{equation}\begin{split}
    \frac{1}{4 \pi} \frac{s+1}{s} \lr{\frac{(s+1)^{s+1}}{s^s}}^n \lr{\pi - \frac{5}{\sqrt{n}}} + {(s+1)n \choose n} - \frac{1}{2} {(s+1)n+1 \choose n} < \\
    A^s_n < \frac{1}{4 \pi} \frac{s+1}{s} \lr{\frac{(s+1)^{s+1}}{s^s}}^n \lr{\pi - \frac{2}{\sqrt{n}}} + {(s+1)n \choose n} - \frac{1}{2} {(s+1)n+1 \choose n}
\end{split}\end{equation}

To help simplify the remaining binomial coefficients, we can note that $1 < \frac{(s+1)n+1}{sn+1} < 2$ for $n \geq 1$, so by Equation \ref{eq:absorption},
\begin{equation}\begin{split}
    0 < {(s+1)n \choose n} - \frac{1}{2} {(s+1)n+1 \choose n} < \frac{1}{2} {(s+1)n \choose n}
\end{split}\end{equation}

Applying this gives
\begin{equation}\begin{split}
    \frac{1}{4 \pi} \frac{s+1}{s} \lr{\frac{(s+1)^{s+1}}{s^s}}^n \lr{\pi - \frac{5}{\sqrt{n}}} < A^s_n < \frac{1}{4 \pi} \frac{s+1}{s} \lr{\frac{(s+1)^{s+1}}{s^s}}^n \lr{\pi - \frac{2}{\sqrt{n}}} + \frac{1}{2} {(s+1)n \choose n}
\end{split}\end{equation}

This completes the proof for the lower bound, but we still need to simplify the upper bound. Applying the bounds from Lemma \ref{lem:stanica_bounds} to eliminate the ${(s+1)n \choose n}$ term yields
\begin{equation}\begin{split}
    A^s_n < \frac{1}{4 \pi} \frac{s+1}{s} \lr{\frac{(s+1)^{s+1}}{s^s}}^n \lr{\pi - \frac{2}{\sqrt{n}}} + \frac{1}{2 \sqrt{2\pi}} \sqrt{\frac{s+1}{s}} \lr{\frac{(s+1)^{s+1}}{s^s}}^n \cdot \frac{1}{\sqrt{n}} = \\
    \frac{1}{4\pi} \frac{s+1}{s} \lr{\frac{(s+1)^{s+1}}{s^s}}^n \lr{ \pi + \lr{-2 + \sqrt{2\pi} \sqrt{\frac{s}{s+1}}} \cdot \frac{1}{\sqrt{n}} }
\end{split}\end{equation}

We would like to simplify the $\frac{1}{\sqrt{n}}$ coefficient. We can note that $\sqrt{2\pi} \sqrt{\frac{s}{s+1}} < \sqrt{2\pi} < 3$, so we can obtain the upper bound
\begin{equation}\begin{split}
    A^s_n < \frac{1}{4\pi} \frac{s+1}{s} \lr{\frac{(s+1)^{s+1}}{s^s}}^n \lr{ \pi + \frac{1}{\sqrt{n}} }
\end{split}\end{equation}

This completes the proof.
\end{proof}

As a corollary, we obtain asymptotics for $A^s_n$,
\begin{corollary}[Total Area Asymptotics in $n$]\label{cor:total_n_asymp} Let $s, n \in \N$, with $s \geq 1$. Then, as $n$ grows large,
\begin{equation}\begin{split}
    A^s_n \sim \frac{1}{4} \frac{s+1}{s} \lr{\frac{(s+1)^{s+1}}{s^s}}^n
\end{split}\end{equation}
\end{corollary}

We will now move on to finding bounds on the average area, $\frac{A^s_n}{F^s_n}$.
\begin{theorem}[Bounds on Average Area]\label{thm:average_area_bounds} Let $s, n \in \N$, with $s \geq 1$ and $n \geq 2$. Then, 
\begin{equation}\begin{split}
    \frac{1}{2\sqrt{2\pi}} \sqrt{\frac{s+1}{s}} \lr{\pi \sqrt{n} - 5} (sn+1) < \frac{A^s_n}{F^s_n} < \frac{1}{2\sqrt{2\pi}} \sqrt{\frac{s+1}{s}} \lr{\pi \sqrt{n} + 2} (sn+1)
\end{split}\end{equation}
\end{theorem}
\begin{proof}
We begin with the lower bound. Taking the lower bound on the total area from Theorem \ref{thm:total_area_bounds} and dividing by $F^s_n = \frac{1}{sn+1} {(s+1)n \choose n}$ gives
\begin{equation}\begin{split}
    \frac{A^s_n}{F^s_n} > \frac{\frac{1}{4 \pi} \frac{s+1}{s} \lr{\frac{(s+1)^{s+1}}{s^s}}^n \lr{\pi - \frac{5}{\sqrt{n}}}}{\frac{1}{sn+1} {(s+1)n \choose n}}
\end{split}\end{equation}

Applying the upper bound from Lemma \ref{lem:stanica_bounds} to the binomial coefficient in the denominator yields the desired lower bound
\begin{equation}\begin{split}
    \frac{A^s_n}{F^s_n} > \frac{1}{2\sqrt{2\pi}} \sqrt{\frac{s+1}{s}} \lr{\pi \sqrt{n} - 5} (sn+1)
\end{split}\end{equation}

For the upper bound, dividing by $F^s_n$ instead requires the lower bound from Lemma \ref{lem:stanica_bounds}, which introduces a factor of $e^{\frac{1}{8n}}$. Dividing the upper bound of Theorem \ref{thm:total_area_bounds} by $F^s_n$ and applying Lemma \ref{lem:stanica_bounds} thus gives
\begin{equation}\begin{split}
    \frac{A^s_n}{F^s_n} < \frac{1}{2\sqrt{2\pi}} \sqrt{\frac{s+1}{s}} \lr{\pi \sqrt{n} + 1} (sn+1) \cdot e^{\frac{1}{8n}}
\end{split}\end{equation}

Finally, since $e^{-x} > 1 - x$, we have $e^{\frac{1}{8n}} < \frac{1}{1 - \frac{1}{8n}} = 1 + \frac{1}{8n-1} \leq 1 + \frac{1}{\pi \sqrt{n} + 1}$, where the last step uses $8n - 1 \geq \pi \sqrt{n} + 1$ for $n \geq 1$. Hence $\lr{\pi \sqrt{n} + 1} e^{\frac{1}{8n}} < \pi \sqrt{n} + 2$, which yields the desired upper bound and completes the proof.
\end{proof}

As with Theorem \ref{thm:total_area_bounds}, we can use the bounds from Theorem \ref{thm:average_area_bounds} to find an asymptotic approximation for the average area.
\begin{corollary}[Average Area Asymptotics in $n$]\label{cor:average_n_asymp} Let $s, n \in \N$, with $s \geq 1$. Then, as $n$ grows large,
\begin{equation}\begin{split}
    \frac{A^s_n}{F^s_n} \sim \frac{1}{2} \sqrt{\frac{\pi}{2}} \sqrt{s(s+1)} \cdot n^{\frac{3}{2}}
\end{split}\end{equation}
\end{corollary}

\begin{remark}
In the specific case $s = 1$, this is the average area for $(n,n)$-Dyck paths,
\begin{equation}\begin{split}
    \frac{A^s_n}{F^s_n} \sim \frac{\sqrt{\pi}}{2} \cdot n^{\frac{3}{2}}
\end{split}\end{equation}
This specific case was derived by Merlini et al. in \cite{merlini} (Theorem 2.1).
\end{remark}

\section{Area Bounds and Asymptotics in $s$}

\subsection{Elementary Binomial Bounds}

While the bounds in the previous section hold for all $s \geq 1, n \geq 2$, they are not particularly good bounds if $s$ is large but $n$ is small. In this section, we will attempt to find bounds that are better for this case. This will allow us to find asymptotics for the total and average area for large $s$, just as we found asymptotics for large $n$ in the previous section.

We will begin by noting that for $s, n \in \N$ with $s \geq 1$ and $n \geq 0$,
\begin{equation}\begin{split}\label{eq:binomial_bounds_s}
    s^n \frac{n^n}{n!} \leq {(s+1)n \choose n} \leq (s+1)^n \frac{n^n}{n!}
\end{split}\end{equation}

These bounds follow by writing the binomial coefficient as a product of $n$ factors,
\begin{equation}\begin{split}
    {(s+1)n \choose n} = \frac{1}{n!} \prod_{i=0}^{n-1} \lr{(s+1)n - i}
\end{split}\end{equation}
and noting that since $0 \leq i \leq n-1 < n$, each factor lies between $sn$ and $(s+1)n$. (For $n = 0$, both products are empty and all three quantities equal $1$.)

These bounds converge asymptotically for large $s$, which cannot be said for the bounds given by Lemma \ref{lem:stanica_bounds}.

One other useful identity is Identity 1.124 from \cite{gould1972combinatorial}, which tells us that
\begin{equation}\begin{split}\label{eq:gould_simplification}
    \sum_{k=0}^{n} \frac{k^k \cdot (n-k)^{n-k}}{k! \cdot (n-k)!} = \sum_{k=0}^{n} \frac{n^k}{k!}
\end{split}\end{equation}

\subsection{Bounds on the Sum and Total Area}

\begin{lemma}[Bounds on the Sum for Large $s$]\label{lem:sum_bounds_s} Let $s, n \in \N$, with $s \geq 1$ and $n \geq 0$. Then, 
\begin{equation}\begin{split}
    s^n \sum_{k=0}^{n} \frac{n^k}{k!} \leq \sum_{k=0}^{n} {(s+1)k \choose k} {(s+1)(n-k) \choose n-k} \leq (s+1)^n \sum_{k=0}^{n} \frac{n^k}{k!}
\end{split}\end{equation}
\end{lemma}
\begin{proof}
Applying Equation \ref{eq:binomial_bounds_s} to the binomial coefficients yields
\begin{equation}\begin{split}
    s^n \sum_{k=0}^{n} \frac{k^k \cdot (n-k)^{n-k}}{k! \cdot (n-k)!} \leq \sum_{k=0}^{n} {(s+1)k \choose k} {(s+1)(n-k) \choose n-k} \leq (s+1)^n \sum_{k=0}^{n} \frac{k^k \cdot (n-k)^{n-k}}{k! \cdot (n-k)!}
\end{split}\end{equation}

We then apply Equation \ref{eq:gould_simplification} to both the upper and lower bounds, which completes the proof.
\end{proof}

We now move on to the total area.
\begin{lemma}[Bounds on Total Area for Large $s$]\label{lem:total_area_bounds_s} Let $s, n \in \N$, with $s \geq 1$ and $n \geq 0$. Then,
\begin{equation}\begin{split}
    \frac{1}{2} \left( s^n \sum_{k=0}^{n} \frac{n^k}{k!} - \frac{s+1}{s} \cdot (s+1)^n \frac{n^n}{n!} \right) \leq A^s_n \leq \frac{1}{2} \left( (s+1)^n \sum_{k=0}^{n} \frac{n^k}{k!} - s^n \frac{n^n}{n!} \right)
\end{split}\end{equation}
\end{lemma}
\begin{proof}
Applying Lemma \ref{lem:sum_bounds_s} to Equation \ref{eq:total_area}, along with the absorption identity (Equation \ref{eq:absorption}), yields
\begin{equation}\begin{split}\label{eq:total_area_bounds_s}
    \frac{1}{2} \left( s^n \sum_{k=0}^{n} \frac{n^k}{k!} - \frac{(s+1)n+1}{sn+1} {(s+1)n \choose n} \right) \leq \\
    A^s_n \leq \frac{1}{2} \left( (s+1)^n \sum_{k=0}^{n} \frac{n^k}{k!} - \frac{(s+1)n+1}{sn+1} {(s+1)n \choose n} \right)
\end{split}\end{equation}

We then note that $1 \leq \frac{(s+1)n+1}{sn+1} < \frac{s+1}{s}$ and apply Equation \ref{eq:binomial_bounds_s} to the remaining binomial coefficients. Since these terms are subtracted, their upper bounds produce our lower bound and their lower bounds produce our upper bound; simplifying completes the proof.
\end{proof}

This gives asymptotics in $s$ for the total area.
\begin{corollary}[Total Area Asymptotics in $s$]\label{cor:total_s_asymp} Let $s, n \in \N$, with $n \geq 1$. Then, as $s$ grows large,
\begin{equation}\begin{split}
    A^s_n \sim \frac{1}{2} s^n \left( \sum_{k=0}^{n} \frac{n^k}{k!} - \frac{n^n}{n!} \right)
\end{split}\end{equation}
\end{corollary}

\subsection{Bounds on the Average Area}\label{sec:s_average_area}

Before deriving bounds and asymptotics for the average area $\frac{A^s_n}{F^s_n}$, we introduce the function that governs them: Ramanujan's $Q$-function,
\begin{equation}\begin{split}\label{eq:q_function}
    Q(n) := \sum_{k=1}^{n} \frac{n!}{(n-k)! \cdot n^k}
\end{split}\end{equation}

$Q(n)$ is best known through the birthday problem: when sampling uniformly with replacement from $n$ items, $Q(n)$ is the expected number of draws made before the first repetition (see \cite{flajolet1995ramanujan}). It is connected to the sums of the previous subsections by the identity
\begin{equation}\begin{split}\label{eq:q_identity}
    Q(n) + 1 = \sum_{k=0}^{n} \frac{n!}{(n-k)! \cdot n^k} = \sum_{j=0}^{n} \frac{n!}{j! \cdot n^{n-j}} = \frac{n!}{n^n} \sum_{j=0}^{n} \frac{n^j}{j!}
\end{split}\end{equation}
where the middle equality substitutes $j = n - k$.

We can now find bounds on the average area, using similar methods to those we used for the total area in Lemma \ref{lem:total_area_bounds_s}.
\begin{lemma}[Bounds on Average Area for Large $s$]\label{lem:average_area_bounds_s} Let $s, n \in \N$, with $s \geq 1$ and $n \geq 0$. Then,
\begin{equation}\begin{split}
    \frac{1}{2} (sn+1) \lr{\frac{s}{s+1}}^n \lr{Q(n)+1} - \frac{1}{2} ((s+1)n+1) \leq \\
    \frac{A^s_n}{F^s_n} \leq \frac{1}{2} (sn+1) \lr{\frac{s+1}{s}}^n \lr{Q(n)+1} - \frac{1}{2} ((s+1)n+1)
\end{split}\end{equation}
\end{lemma}
\begin{proof}
Take Equation \ref{eq:total_area_bounds_s} and divide by $F^s_n$ to obtain
\begin{equation}\begin{split}
    \frac{\frac{1}{2} s^n \sum_{k=0}^{n} \frac{n^k}{k!}}{\frac{1}{sn+1} {(s+1)n \choose n}} - \frac{1}{2} ((s+1)n+1) \leq \frac{A^s_n}{F^s_n} \leq \frac{\frac{1}{2} (s+1)^n \sum_{k=0}^{n} \frac{n^k}{k!}}{\frac{1}{sn+1} {(s+1)n \choose n}} - \frac{1}{2} ((s+1)n+1)
\end{split}\end{equation}

We then apply Equation \ref{eq:binomial_bounds_s} to the remaining binomial coefficients. Since these appear in denominators, the upper bound in Equation \ref{eq:binomial_bounds_s} produces the lower bound below, and vice versa. This gives
\begin{equation}\begin{split}
    \frac{1}{2} (sn+1) \lr{\frac{s}{s+1}}^n \cdot \frac{n!}{n^n} \sum_{k=0}^{n} \frac{n^k}{k!} - \frac{1}{2} ((s+1)n+1) \leq \\
    \frac{A^s_n}{F^s_n} \leq \frac{1}{2} (sn+1) \lr{\frac{s+1}{s}}^n \cdot \frac{n!}{n^n} \sum_{k=0}^{n} \frac{n^k}{k!} - \frac{1}{2} ((s+1)n+1)
\end{split}\end{equation}

Applying Equation \ref{eq:q_identity} completes the proof.
\end{proof}

We can now obtain asymptotics in $s$ for the average area.
\begin{corollary}[Average Area Asymptotics in $s$]\label{cor:average_s_asymp} Let $s, n \in \N$, with $n \geq 1$. Then, as $s$ grows large,
\begin{equation}\begin{split}
    \frac{A^s_n}{F^s_n} \sim \frac{1}{2} sn \cdot Q(n)
\end{split}\end{equation}
\end{corollary}

\begin{remark}\label{rem:urn_ev}
By the birthday interpretation of $Q(n)$, Corollary \ref{cor:average_s_asymp} states that, for large $s$, the average area is approximately $\frac{sn}{2}$ times the expected number of draws made before the first repetition. There is likely some kind of combinatorial explanation for this connection, but it is not immediately clear to us.
\end{remark}

\section{Conclusions and Future Directions}

While we have addressed the $(sn,n)$ case, the area formulas for Dyck paths in the more general case of $(m,n)$-shaped rectangles are still unknown. One of the primary issues with approaching this case is that touchpoint decompositions are more difficult to apply, since the number of touchpoints differs from case to case. For example, if $m$ and $n$ are coprime, there are \textit{no} touchpoints, which is particularly difficult to approach.

We note in Remark \ref{rem:chess_guy_conjecture} that we have proved a conjecture of Vaclav Kotesovec regarding the asymptotics of the sum $\sum_{k=0}^{n} {(s+1)k \choose k} {(s+1)(n-k) \choose n-k}$. However, there is still work to be done here, as the corresponding OEIS entries (\cite{oeisA006256}, \cite{oeisA078995}) record refined conjectures, due to Benoit Cloitre and Kotesovec, stating that for $s = 2$, we have
\begin{equation}\begin{split}
    \sum_{k=0}^{n} {3k \choose k} {3(n-k) \choose n-k} = \frac{3}{4} \lr{\frac{3^3}{2^2}}^n \lr{1 + \frac{c}{\sqrt{n}} + o\lr{\frac{1}{\sqrt{n}}}} \text{ with } c = \frac{2}{3 \sqrt{3\pi}}
\end{split}\end{equation}

and that for $s = 3$,
\begin{equation}\begin{split}
    \sum_{k=0}^{n} {4k \choose k} {4(n-k) \choose n-k} = \frac{4}{6} \lr{\frac{4^4}{3^3}}^n \lr{1 + \frac{c}{\sqrt{n}} + o\lr{\frac{1}{\sqrt{n}}}} \text{ with } c = \frac{2 \sqrt{2}}{3 \sqrt{3\pi}}
\end{split}\end{equation}

This implies the existence of more precise bounds and asymptotics than those we derived. The kernel-method framework of Banderier and Gittenberger (\cite{banderier2006area}), which produces full asymptotic expansions for area moments in the aperiodic case, is a natural candidate tool for such refinements. It is unclear how this conjecture extends to higher values of $s$. For $s = 4$, $s = 5$, and $s = 6$, numerical estimation of the $c$ value with Mathematica yields values of $c \approx 0.356823$, $c \approx 0.388460$, and $c \approx 0.410386$, but it is not obvious what the precise values of these should be. The fact that $c$ has such a clean closed form for $s = 2$, $s = 3$ (and $s = 1$, where the value is $c = 0$) seems to suggest that $c$ may have a similarly neat form for $s > 3$, but that form is unclear. If there is such a neat form, then there is immediately a question of whether the further expansion also has a neat form. For example, it is not difficult to imagine that instead of $1 + \frac{c}{\sqrt{n}} + o\lr{\frac{1}{\sqrt{n}}}$, we might be able to find some series of the form $1 + \frac{c_1}{\sqrt{n}} + \frac{c_2}{n} + \frac{c_3}{n^{\frac{3}{2}}} + \cdots$. If such a series existed, we would be able to approximate the sum $\sum_{k=0}^{n} {(s+1)k \choose k} {(s+1)(n-k) \choose n-k}$ to arbitrary precision for any $s$ and $n$.

As noted in Remark \ref{rem:urn_ev}, it seems as though there should be some kind of combinatorial interpretation of the asymptotics in $s$ for the average area, due to its connection with Ramanujan's $Q$-function and the birthday problem. Investigation into said combinatorial interpretation represents one interesting area of further study.

While we focus on the expectation of the area, one promising direction of future research is to examine the properties of the distribution of the areas. This topic has been more thoroughly studied for $(n,n)$-Dyck paths (see \cite{catalan_area_distribution} for more information). Many of the papers relating to this topic are primarily in the area of probability, rather than combinatorics, as the asymptotic area of an $(n,n)$-Dyck path can be viewed as the area of a particular Brownian excursion called the Bernoulli excursion (see Takács \cite{bernoulli_excursion}). Extending these results to the $(sn,n)$ case would be extremely useful.

\printbibliography

\end{document}